\documentclass{amsart}

\usepackage{amsmath}
\usepackage{amsfonts}
\usepackage{amssymb}
\usepackage{amscd}
\usepackage{amsthm}
\usepackage[all]{xy}
\usepackage[pdftex]{graphicx,color}
\usepackage{epstopdf}
\usepackage{epsfig}

\newcommand{\R}{\mathbb{R}}
\newcommand{\RP}{\mathbb{RP}}
\newcommand{\C}{\mathbb{C}}

\newcommand{\ra}{\rightarrow}

\def\ra{\rightarrow}

\def\({\big(}
\def\){\big)}

\def\rp2{\mathbb{R}P^2}
\def\rpn{\mathbb{R}P^n}
\def\rpm{\mathbb{R}P^{n-1}}
\def\pglnone{PGL(n+1,{\mathbb R})}
\def\pglnonec{PGL(n+1,{\mathbb C})}
\def\glnonec{GL(n+1,{\mathbb C})}

\def\pglnc{PGL(n,{\mathbb C})}

\def\glnc{GL(n,{\mathbb C})}

\def\slnc{SL(n,{\mathbb C})}

\def\slnone{SL(n+1,{\mathbb R})}

\def\slnonepm{SL_{\pm}(n+1,{\mathbb R})}
\def\slnonepmc{SL_{\pm}(n+1,{\mathbb C})}
\def\slnonec{SL(n+1,{\mathbb C})}

\def\pslnc{PSL(n,{\mathbb C})}
\def\sononec{SO(n+1,{\mathbb C})}
\def\sonc{SO(n,{\mathbb C})}

\def\psononec{PSO(n+1,{\mathbb C})}
\def\psonc{PSO(n,{\mathbb C})}

\def\ononec{O(n+1,{\mathbb C})}
\def\onc{O(n,{\mathbb C})}
\def\pononec{PO(n+1,{\mathbb C})}
\def\ponc{PO(n,{\mathbb C})}

\def\slnpm{SL_{\pm}(n,{\mathbb R})}
\def\rn{{\mathbb R}^n}
\def\rnone{{\mathbb R}^{n+1}}

\def\ponone{PO(n,1)}
\def\sonone{SO(n,1)}
\def\sn{S^n}

\newtheorem{theorem}{Theorem}[section]
\newtheorem{prop}[theorem]{Proposition}
\newtheorem{cor}[theorem]{Corollary}
\newtheorem{lemma}[theorem]{Lemma}
\theoremstyle{definition}
\newtheorem{dfn}[theorem]{Definition}

\newtheorem{eg}{Example}
\newtheorem{remark}[theorem]{Remark}

\begin{document}
\title{The marked length spectrum of a projective manifold or orbifold.}
\author{Daryl Cooper}

\curraddr{Department of Mathematics, University of California at Santa Barbara}
\email{cooper@math.ucsb.edu}

\author{Kelly Delp}
\curraddr{Department of Mathematics, Buffalo State College}
\email{kelly.delp@gmail.com}


\begin{abstract}
A strictly convex real projective orbifold is equipped with a natural Finsler metric called the Hilbert metric.  In the case that the projective structure is hyperbolic, the Hilbert metric and the hyperbolic metric coincide.  We prove that the marked Hilbert length spectrum determines the projective structure only up to projective duality.  A corollary is  the existence of non-isometric diffeomorphic strictly convex projective manifolds (and orbifolds) that are isospectral.  The corollary follows from work of Goldman and Choi, and Benoist.
\end{abstract}

\maketitle

\section{Introduction}

Let $M$ be a  Riemannian manifold.  The marked length spectrum of $M$ is the function $\ell: \pi_1(M) \ra \R$ which assigns to each free homotopy class of loops the infimum of the length of loops in the class.  An outstanding problem in Riemmannian geometry is the question of marked length spectrum rigidity for nonpositively curved manifolds: if a compact manifold $M$ supports two nonpositively curved metrics $g$ and $g^\prime$ with the same marked length spectrum, is $(M,g)$ isometric to $(M,g^\prime)$?  This conjecture, which is attributed to Katok, appears in  \cite{burns_katok}.   For non-positively curved surfaces, the marked length spectrum determines the  metric \cite{croke},  \cite{otal}.   In \cite{bcg_rigidity}  the
authors prove the conjecture in the case of negatively curved locally symmetric manifolds.

In this paper we study compact, strictly convex, real projective orbifolds. We defer precise definitions to the next section. These objects include all hyperbolic manifolds. The Hilbert metric on such an $n$-orbifold, $Q,$ assigns  a length to every element of the orbifold fundamental group $\pi_1^{orb}Q.$ The Hilbert-length spectrum is the resulting map $\ell:\pi_1^{orb}Q\rightarrow{\mathbb R}.$ Such a structure has a unique dual which is another such structure. Our main result is the following:

\begin{theorem}\label{maintheorem} Let $Q^n$ be a compact, strictly convex, real projective orbifold. Two distinct structures on $Q$ have the same marked Hilbert-length spectrum iff each structure is the projective dual of the other.
\end{theorem}

It is known that:
\begin{theorem}\label{selfdualishyp}
Let $Q^n$ be a compact, strictly convex, real projective orbifold.  Then $Q$ is self dual iff $Q$ is real hyperbolic.
\end{theorem}

Together these provide examples of iso-spectral structures which are not isometric. However, in these cases, the Hilbert metric is not Riemannian, only Finsler. 

\medskip

We now sketch the proof of the main theorem, making a few simplifying assumptions without mention.  We are given two monomorphisms of $G=\pi_1^{orb}Q$ into $\pglnone.$  Combining these gives an embedding of $G$ onto $\Gamma\subset\pglnone\times\pglnone.$ Let $\overline{\Gamma}$ denote the Zariski closure. It is an algebraic group. By a beautiful theorem of Benoist,  the projection of $\overline{\Gamma}$ onto each factor is either surjective or a conjugate of $\ponone.$ The latter is the hyperbolic case. In the former we now use the hypothesis on the length functions to show $\overline{\Gamma}$ is contained in a {\em proper} subvariety: in essence the translation length in the Hilbert metric of a matrix in $\pglnone$ is the logarithm of the ratio of two eigenvalues,  and equality of ratios of eigenvalues of a pair of matrices is an algebraic condition. Since $\pglnone$ is a simple Lie group, it follows that both coordinate projections of $\overline{\Gamma}$ are isomorphisms. Combining the inverse of one projection with the other projection gives an automorphism of $\pglnone$ taking one embedding of $G$ to the other. The theorem follows from the  classification of automophisms of $\pglnone$ 
namely the outer automorphism group is order two generated by the {\em global Cartan involution} given by the composition of transpose and inverse. This sends a structure to its dual. In the main proof we work over ${\mathbb C}$ rather than over ${\mathbb R}$ since this makes the algebraic geometry simpler. This completes the sketch.

\medskip
The main theorem is essentially due to Inkang Kim, \cite{kim}. However he omitted the possibility of dual structures and the consequent discussion of self-dual structures. Our proof follows his in outline, but at a crucial point we use elementary algebra in place of a more sophisticated argument using buildings.   We also extend Kim's result to compact orbifolds. 
In other work we study limits at infinity of these projective structures and we find that at infinity projective duality becomes isometry. This leads to some beautiful phenomena which are further explored in that work.

\medskip
The authors thank Darren Long and Morwen Thistlethwaite for helpful discussions.  This work was partially supported by NSF grant DMS-0706887. We also thank the referee for their careful reading and useful comments.  The first author also thanks the mathematics department at the University of Capetown, and especially David Gay, for hospitality during the preparation of this paper.

\section{Projective Manifolds and Orbifolds.}\label{section:background}

\subsection{Projective structures}

Real projective $n$-space $\rpn$ has automorphism group the projective general linear group $\pglnone,$ an element of which is called a {\em projective transformation.} If $P\cong\rpm$ is a codimension-1 projective subspace in $\rpn$ the complement $\rpn\setminus P$ is called an {\em affine patch.} An affine patch can be naturally identified with $\rn$ in a way which is unique up to affine transformations of $\rn.$  A subset $\Omega\subset\rpn$ is called a {\em strictly convex subset} if it has closure, $\overline{\Omega},$ which is a compact strictly convex subset of some affine patch $\rn.$ This means that $\partial\overline{\Omega}$ does not contain any line segment of positive length.

A {\em strictly convex real projective orbifold,} $Q,$ is the quotient $Q=\Omega/\Gamma$ where $\Omega$ is an open strictly convex set in $\rpn$ and $\Gamma$ is a discrete subgroup of $\pglnone$ which preserves $\Omega.$ The {\em orbifold fundamental group} $\pi_1^{orb}Q$ is defined to be $\Gamma.$  The {\em orbifold universal cover} is $\tilde{Q}=\Omega$ and the universal orbifold-covering projection is the quotient map $\pi:\tilde{Q}\rightarrow Q$ (see \cite{kapovich} p 139 for a more general definition). Observe that the orbifold fundamental group, $\pi_1^{orb}Q,$ is the group of homeomorphisms that cover the identity map. If $\Gamma$ acts freely on $\Omega$ then $Q$ is a {\em projective manifold.}

The natural map $\slnone\rightarrow\pglnone$ is an isomorphism when $n$ is even.   However, in all dimensions $Aut(\Omega)$ lifts to $\slnonepm.$  It is often easier for computations to work with  the latter group.   We include the following well known argument for the convenience of the reader. 
\begin{lemma}\label{lemma:lift}
Let $\Omega$ be a strictly convex bounded subset of an affine patch in $\rpn.$ Then there is a lift of $Aut(\Omega)$ to $\slnonepm$.
\end{lemma}

\begin{proof}
 Let $\pi:S^n\rightarrow \rpn$ denote the double cover.   Then $\pi^{-1}(\Omega)=\Omega_+\cup\Omega_-$  has two connected components in $S^n.$ An element of $Aut(\Omega)$ corresponds to a pair of matrices $\pm A\in\slnonepm$ which act on $S^n$ and permute these two components.  It follows that exactly one of this pair acts as the identity permutation. This defines the lift.
\end{proof}

From now on, in this paper we will use the term {\em projective orbifold} to mean strictly convex real projective orbifold. Furthermore, we will assume that it is of the form $\Omega/\Gamma$ with $\Omega$ a bounded strictly convex open subset of $\mathbb{R}^n$ and  $\Gamma$ a discrete subgroup of either $\pglnone$ or $\slnonepm$.

An {\em orbifold isomorphism} between projective orbifolds $Q=\Omega/\Gamma$ and $Q'=\Omega'/\Gamma'$ is a map $f:Q\rightarrow Q'$ which is covered by a homeomorphism $\tilde{f}:\Omega\rightarrow \Omega'.$ Given $f,$ the choice of $\tilde{f}$ is unique up to post-composition with an element of $\Gamma'.$  Such a choice induces an isomorphism $f_*:\pi_1^{orb}Q\rightarrow\pi_1^{orb}Q'$ of groups given by $f_*(\gamma)=\tilde{f}\circ\gamma\circ\tilde{f}^{-1}.$  The indeterminacy of this choice means $f_*$ is only defined up to conjugacy.

The map $f$ is a {\em projective isomorphism} if $\tilde{f}$ is the restriction of a projective transformation. In this case $\pi_1^{orb}Q$ and $\pi_1^{orb}Q'$ are conjugate (via $f_*$) subgroups of $PGL(n+1,\mathbb{R}).$

To define the space of projective structures on a given orbifold we choose one such structure as a base point. Given a projective orbifold $Q$ we define the {\em space, $C(Q),$ of strictly convex real  projective structures on $Q$} as follows. Let $\tilde{C}(Q)$ denote the set of all triples $(f,\Omega,\Gamma)$ where $R=\Omega/\Gamma$ is a projective orbifold and $f:Q\rightarrow R$ is an {\em orbifold} isomorphism called a {\em marking} which is not assumed to be a projective isomorphism. 

An element of $\tilde{C}(Q)$  determines a group monomorphism (up to conjugacy) $$hol:\pi_1^{orb}Q\rightarrow \pglnone$$ called the {\em holonomy} of the structure and is defined as $hol=f_*.$ The image of the holonomy is (up to conjugacy) $\Gamma.$  

We define an equivalence relation $\sim$ on $\tilde{C}(Q)$ by $(f_0,\Omega_0,\Gamma_0)\sim (f_1,\Omega_1,\Gamma_1)$  if there is a projective isomorphism $k:\Omega_0/\Gamma_0\rightarrow \Omega_1/\Gamma_1$ that conjugates $hol_0$ to $hol_1.$ This means $hol_1=\tilde{k}\circ hol_0\circ \tilde{k}^{-1}$ where $\tilde{k}:\Omega_0\rightarrow\Omega_1$ covers $k.$  Then we define $C(Q)=\tilde{C}(Q)/\sim$ to be the set of all equivalence classes $[(f,\Omega,\Gamma)].$

\begin{prop}[$\Gamma$ determines $\Omega$]\label{groupgivesomega} If $(f,\Omega,\Gamma)$ and $(f',\Omega',\Gamma)$ are in $\tilde{C}(Q)$ then $\Omega=\Omega'.$ \end{prop}
\begin{proof}  Let $K$ denote the closure of the set of attracting fixed points of proximal (see defintion \ref{dfn:proximal}) elements of $\Gamma.$ Since $\Omega/\Gamma$ is compact  $K=\partial\Omega.$  The domain $\Omega$ is strictly convex and therefore $K$ is contained in an affine patch of $\rpn$.  This implies $K$ separates $\rpn$ into two connected components, one of which is $\Omega$ and the other is homeomorphic to the normal bundle of a codimension one projective subspace, thus not a cell. Thus $\Omega$ is characterized as the component of $\rpn\setminus K$ which is a cell.\end{proof}

\begin{cor}[holonomy determines structure]\label{holgivesstructure} The map $$Hol:C(Q)\rightarrow Hom(\pi_1^{orb}Q,\pglnone)/\pglnone$$ which sends a structure to the conjugacy class of its holonomy is injective.
\end{cor}
\begin{proof} For $i=1,2$ suppose that $hol_i:\pi_1^{orb}Q\rightarrow\Gamma_i$ is the holonomy of $[(f_i,\Omega_i,\Gamma_i)]\in C(Q).$  By \ref{groupgivesomega} the images of $hol_i$ determines $\Omega_i.$ If $hol_1$ and $hol_2$ are conjugate by $A\in \pglnone$ using \ref{groupgivesomega} again shows that $A$ sends $\Omega_1$ to $\Omega_2.$  Thus the two structures are equivalent.
\end{proof}

\begin{theorem}[Benoist, \cite{benoist3}]\label{benoistvariety} Suppose that $Q$ is a closed strictly convex real projective orbifold. Then the image of $Hol$ in $Hom(\pi_1^{orb}Q,\pglnone)/\pglnone$  is a union of path  components of  the Euclidean topology.
\end{theorem}

A natural question is: to what extent does the orbifold fundamental group determine a convex projective structure? There is the issue of {\em moduli}: there may be different structures on a given isomorphism type of orbifold. But there is a second issue.
The {\em Borel conjecture} is that two closed aspherical manifolds with isomorphic fundamental groups are homeomorphic. 

The Borel Conjecture holds for word hyperbolic groups, \cite{bartels}. The fundamental group of a closed strictly-convex projective manifold is word hyperbolic \cite{benoist1}. Hence the Borel conjecture holds for strictly-convex projective manifolds. 

The  results above give a weak version of the Borel conjecture for orbifolds, but with isomorphism (orbifold-diffeomorphism)   in place of homeomorphism:
\begin{cor} If $G$ is a group then there are at most finitely many isomorphism classes of closed connected orbifolds $Q$ that admit a strictly convex projective structure such that $\pi_1^{orb}Q\cong G.$\end{cor}
\begin{proof} The holonomy determines the isomorphism type of a projective orbifold. If two projective orbifolds are part of a one parameter family then they are isomorphic. Theorem \ref{benoistvariety}  implies there is at most one isomorphism type of $Q$ for each path component of  $Hom(\pi_1^{orb}Q,\pglnone)/\pglnone$ in the Euclidean topology. Since a real algebraic variety has only finitely many path components, the result follows.\end{proof}

\subsection{The Hilbert Metric}
\begin{dfn}[Hilbert metric]

Let $\Omega$ be a  bounded convex open subset of $\R^n$.   Let $p,q$ be distinct points in  $\Omega$ and  let $p_{\infty}$ and $q_{\infty}$ be the points on  the boundary of $\Omega$ colinear with $p$ and $q$ and ordered, so that $p$ separates $p_{\infty}$ from $q$.  The Hilbert distance between $p$ and $q$ is given by the logarithm of the cross ratio of the points $(p_{\infty},p,q,q_{\infty})$ namely:

\[d_H(p,q) = \log\left({\frac{\|q-p_{\infty}\|}{\|p-p_{\infty}\|}\frac{\|p-q_{\infty}\|}{\|q-q_{\infty}\|}}\right)\]
\end{dfn}

The usual argument adapts to show that the straight line with endpoints $p,q$ has minimal length, and in the case that $\Omega$ is strictly convex,  this is the unique path of minimal length with these endpoints \cite{delaharpe}.  
 
Let $Aut(\Omega)$ denote the group of projective transformations which map $\Omega$ onto itself.  Following Benoist we say a discrete subgroup $\Gamma\subset Aut(\Omega)$ {\em divides $\Omega$} if the quotient $Q=\Omega/\Gamma$ is compact.   Since cross ratio is preserved by projective transformations, $Aut(\Omega)$ (and thus the subgroup $\Gamma$) acts on $\Omega$ by isometries of the Hilbert metric.  It follows that the projective orbifold $Q$ inherits a metric from $d_{\Omega}$ called the {\em Hilbert metric} on $Q.$ The holonomy provides an identification $\pi_1^{orb}Q\equiv\Gamma.$

Given $g\in Aut(\Omega)$ the {\em Hilbert translation length} of $g$ is 
$$\ell(g)=\inf_{x\in\Omega} d_{\Omega}(x,gx).$$ Since $Aut(\Omega)$ acts by isometries it follows that conjugate elements have the same translation length. If $Q=\Omega/\Gamma$ is a strictly convex real projective orbifold we define the {\em Hilbert translation length function} $\ell: \pi_1^{orb}Q \ra {\mathbb R}$ where $\ell(g)$ is the Hilbert length of  $g$.  If $Q$ is a compact manifold then $\ell(g)$ is the length of the unique geodesic loop representative in the free homotopy class.
 
 From now on we assume $\Omega$ is strictly convex. The geodesics in the Hilbert metric on $\Omega$ are then (subarcs of projective) lines.  An isometry $g\in Aut(\Omega)$ is called {\em hyperbolic} if it preserves a line in $\Omega$ called the {\em axis}, and otherwise $\gamma$  fixes some point in $\Omega$ and is called {\em elliptic.}   When $g$ is hyperbolic we can compute $\ell(g)$ from the eigenvalues of $g$.  
 
 \begin{dfn}\label{dfn:proximal}
 A matrix $M\in\slnonepm$  is \textit{semi-proximal} if it has an attracting fixed point $p$ in $\RP^n;$    algebraically this just means $M$ has a unique eigenvalue of largest modulus, it has algebraic multiplicity $1,$ and  is real. The matrix $M$ is {\em proximal} if both $M$ and $M^{-1}$ are semi-proximal and, in addition, the eigenvalues of smallest and largest modulus have the same sign.     
\end{dfn}

Theorem 1.1 of \cite{automorphs_benoist} implies that hyperbolic isometries are proximal.    
 
 \begin{lemma}\label{lemma:Hilbert_length}
 Suppose $g\in Aut(\Omega)$ is hyperbolic and let $M \in \slnonepm$ be the proximal matrix which represents $g$.  Let $\lambda_+$ and $\lambda_{-}$ be the eigenvalues of $M$ with largest and smallest absolute value. Then the Hilbert translation length of $g$ is given by $ log(\lambda_+/\lambda_-)$.
 \end{lemma}
\begin{proof}  Let $p_{\pm}$ be the attracting fixed points of $g^{\pm1}.$ These two points are in the boundary of $\Omega$ in $\rpn$ and correspond to the eigenvectors $v_{\pm}\in\rnone$ for the eigenvalues $\lambda_{\pm}.$ Let $A$ be the line segment in $\Omega$ with endpoints $p_{\pm}.$ Then $A$ is the axis of $g.$ The Hilbert metric on the interior of $A$ equals the restriction of the Hilbert metric on $\Omega$ to $A.$ It follows that the translation length of $g$ in $\Omega$ equals the translation length of the restricton of $g$ to $A,$ which  is 
$d_A(q,gq)$ for any choice of a point $q$ in the interior of $A.$ This reduces the problem of computing the translation length of $g$ to the case of ${\mathbb R}P^1.$ The action of $g$ on $A$ is determined by the restriction of $g$ to the invariant 2-dimensional subspace of $\rnone$ spanned by $v_{\pm}.$ This action is given by a mobius transformation fixing two points. The Hilbert metric on $A$ is the one dimensional hyperbolic metric, and an easy calculation (or see \cite{beardon}) gives that the translation length of a M\"obius transformation in this metric is as stated.\end{proof}

It follows from the next result that the notion of (projective) equivalence used in the definition of the space of strictly convex projective structures, $C(Q),$ on a specified orbifold $Q$ can be taken to be isometry of Hilbert metrics which preserve the marking.

\begin{prop}[Isometric implies equivalent]\label{theorem:not_isometric}
If $Q$ and $Q'$ are strictly convex projective orbifolds with isometric Hilbert metrics then they  are projectively equivalent.
\end{prop}

\begin{proof} 
Let $Q=\Omega/\Gamma$ and $Q'=\Omega'/\Gamma'.$ First we show that an isometry  $f:Q\rightarrow Q'$  is an orbifold isomorphism and therefore covered by a map $h:\Omega\rightarrow\Omega'$ so the following diagram commutes.

\[\xymatrix{\Omega \ar[r]^{h} \ar[d]_{\rho} & \Omega'  \ar[d]^{\rho^\prime} 
 \\ \Omega/\Gamma   \ar[r]_{f}   &
\Omega^\prime/ \Gamma^\prime  \\}\]
 
 Let $\Sigma Q$ denote the singular set of $Q$ and choose a small ball $U \subset Q - \Sigma Q$ so that $f(U) \subset Q^\prime - \Sigma Q^\prime$ and the quotient map $\rho: \Omega \rightarrow \Omega/\Gamma$ restricted to a lift $\tilde{U}$ of $U$ is an isometry.  Choose a lift $\tilde{U^\prime}$ of $f(U)$.  The map $h=(\rho^\prime|_{\tilde{U^\prime}})^{-1}\circ f \circ \rho$ is an isometry from $\tilde{U}$ to $\tilde{U^\prime}$. Let $\tilde{x}$ be a point in $\tilde{U}$ and $\tilde{y} = h(\tilde{x})$.      For any line $\ell$ through $\tilde{x}$ there is a corresponding line $\ell^\prime$  in $\Omega^\prime$ which passes through $\tilde{y}$ and contains $h(\tilde{U} \cap \ell)$.   An isometry between lines is determined by its restriction to an interval, so for every $\ell$ through $\tilde{x}$, extend $h$ isometrically from $\ell$ to $\ell^\prime$.   Since $\Omega$ is a convex ball in $\R^n$, the set of all lines $\ell-\tilde{x}$ foliate $\Omega-\tilde{x}$.   Therefore, we have a continuous extension of $h$ to all of $\Omega$ that, when restricted to each line through $\tilde{x}$, is an isometry.

\noindent
{\bf Claim.} $\rho^\prime \circ h = f \circ \rho$

Assuming this, it follows that $h$ covers $f$ and therefore, by definition, $f$ is an orbifold isomorphism. Moreover, $f$ is an isometry which implies the map $h$ is an isometry of the Hilbert domains $\Omega$ and $\Omega^\prime$.  Since the domains are strictly convex, an isometry is the restriction of a projective transformation (see proposition 3 of \cite{delaharpe}).  Therefore the orbifolds $Q$ and $Q^\prime$ are projectively equivalent. It only remains to prove the claim. 

We first discuss the local structure of the singular locus when there are reflections. Locally a projective orbifold is $V/G$ where $V$ is a open set in projective space and $G$ is a finite group of projective transformations that preserve $V.$ The singular locus in $V/G$ is the image of the union of the subsets of  fixed points of non-trivial elements in $G.$ This is the intersection with $V$ of a finite set of projective subspaces. The fixed subset of a projective transformation $T$ has codimension $1$ iff $T$ is a {\em reflection} which means it is conjugate to $\pm diag(-1,1,1,\cdots,1).$  A {\em mirror point} in an orbifold is a point with local group of order two generated by a reflection. The {\em mirror} of an orbifold is the collection of all mirror points. The mirror is an open subset of the singular locus. If the mirror is not empty it has codimension $1.$ The singular locus minus the mirror has codimension at least $2.$

A {\em mirror geodesic} in $Q$ is a map $\ell:{\mathbb R}\rightarrow Q$ that is a composition $\ell=\rho\circ\tilde{\ell}$ where $\tilde{\ell}:{\mathbb R}\rightarrow \Omega$ is a Hilbert  geodesic such that  $\ell$ is a path in $Q$ which only intersects the singular locus at  mirror points (where it bounces off the mirror like a light ray.) In addition we require that a mirror geodesic is not contained in the singular locus. By considering the local structure of $V/G$ it follows that the set of singular points on a mirror geodesic is discrete and that mirror geodesics are dense in $Q.$ In particular a geodesic which is disjoint from the singular locus is a mirror geodesic.

If $\gamma$ is a mirror geodesic in $Q$ then $f\circ\gamma$ is a mirror geodesic in $Q'.$ The reason is that in a Hilbert domain there is a unique geodesic connecting any two points, and it is a segment of a projective line. Thus, away from the singular locus, mirror geodesics are sent to mirror geodesics by an isometry. This property also holds true near a mirror point $x$ on a mirror geodesic $\gamma.$   This is seen by working in an orbifold neighborhood $V/G.$  Near $x$ there is a small segment $\delta$ of $\gamma$ containing $x$ in its interior which is length minimizing among all paths with with the same endpoints as $\delta$ subject to the constraint that the path contains a mirror point. 

Suppose $\ell$ is a line through $\tilde{x}$ such that $\rho\circ\ell$ is a mirror geodesic in $Q$ and $\rho^\prime\circ h\circ\ell$ is mirror geodesic in $Q'.$ The set  of all such $\ell$ is dense in $\Omega.$ Since $\rho\circ\ell$ is a mirror geodesic it follows from the above discussion that $f\circ\rho\circ\ell$ is a mirror geodesic. This mirror geodesic coincides with $\rho^\prime\circ h\circ\ell$ in $U$ and therefore they are equal everywhere. It follows that the claim holds on the dense subset of $\Omega$ consisting of all such $\ell.$ By continuity the claim holds everywhere.
  \end{proof}

\section{The Root Ratio polynomial}\label{section:root_ratio}
The main result of this section is \ref{samelength} where we show the subset of $\slnpm\times\slnpm$ consisting of pairs of proximal matrices with the same Hilbert translation length is not Zariski dense. In view of \ref{lemma:Hilbert_length}, accomplishing  this involves a discussion of ratios of pairs of roots of a polynomial.

\begin{dfn}\label{dfn:root_ratio_poly} Let  $p(x)\in{\mathbb C}[x]$  be a degree $n$ polynomial with roots $\{\alpha_i\in{\mathbb C}:\ 1\le i\le n\ \}$ counted with multiplicity. 
A complex number $r$ is a \textit{root ratio} of  $p(x)$ if  $r=\alpha_i/\alpha_j$, where $\alpha_j$ is a non-zero root and $i\neq j$.

\medskip
Let $p(x)= x^n+a_1x^{n-1}+\cdots+a_n$  be a  monic polynomial of degree  $n\ge2$ in the ring ${\mathbb C}[x]$ with roots $\alpha_i$ counted with multiplicity.  The \textit{root ratio polynomial} of $p$ is:
\[R_p(r) = \prod_{i \neq j} (\alpha_i r - \alpha_j)\ \in\ {\mathbb C}[r]\]
\end{dfn}

\begin{prop}\label{rootratios} If $p$ is a polynomial of degree at least $2$ and $p(0)\ne0$  then the zeros of $R_p$ are exactly the root ratios of $p$.
\end{prop}
Since the product of the roots of $p(x)$ is $(-1)^na_n$ it follows that, \[R_p(r) = a_n^{n-1}r^{n(n-1)} + \textrm{lower order terms}.\]  Furthermore, since the coefficients of $R_p$ are symmetric polynomials in the $\alpha_i$, they are polynomials in the coefficients $a_i$. 

\medskip
\begin{eg}\label{eg_root_ratio}
The root ratio polynomial of $p(x)=x^2+bx+c$ is $$R_p(r)=c r^2 +(2c-b^2)r +c$$ and that of $p(x)=x^3+b x^2 + c x +d$ is 
$$\begin{array}{rcl}
R_p(r) & = &d^2 r^6 + (3d^2-bcd) r^5 +(c^3+b^3d-5bcd+6d^2) r^4\\
 & &  +(-b^2c^2+2c^3+2b^3d - 6bcd + 7d^2) r^3\\ & & +(c^3+b^3d-5bcd+6d^2)r^2
  +(3d^2-bcd)r + d^2.
 \end{array}$$
\end{eg}
By a {\em general polynomial} we mean a polynomial whose coefficients are independent transcendentals. Subsequently we may replace these by certain complex numbers. To this end we will let $K$ denote a polynomial ring over ${\mathbb C}$ and consider polynomials in $K[x]$ thought of as polynomials in the variable $x$ with coefficients in $K.$ A ${\mathbb C}$-module homomorphism $\theta:K\rightarrow{\mathbb C}$ will be called a {\em specialization.}  Given a specialization, there is a unique extension to a ${\mathbb C}$-module homomorphism between the polynomial rings $K[x]$ and ${\mathbb C}[x]$, which for convenience we also denote by $\theta$.

Given two general polynomials $$A(x)=a_0x^m+a_{1}x^{m-1}+\cdots a_m\qquad B(x)=b_0x^n+b_{1}x^{n-1}+\cdots+b_n$$ the {\em resultant} $R(A(x),B(x);x)\in K$ is a polynomial in the $m+n+2$ coefficients of these two polynomials obtained by eliminating $x,$ see \cite{lang}. If $$A(x)=a_0\prod_{i=1}^m(x-\alpha_i)\qquad\text{and}\qquad B(x)=b_0\prod_{j=1}^n(x-\beta_i).$$Then
 $$Res(A(x),B(x);x)=a_0^nb_0^m\prod_{i=1}^m\prod_{j=1}^n(\alpha_i-\beta_j).$$

From this one sees that the coefficients of the resultant are symmetric polynomials of the $\alpha_i$ and  the $\beta_j.$ The elementary symmetric functions generate the ring of all symmetric polynomials,  therefore the terms in the resultant are polynomial functions of the coefficients of $A(x)$ and $B(x).$ The crucial property for us is:

\begin{prop}\label{resultantcases} Suppose $A(x),B(x)\in K[x]$  and $\theta:K \rightarrow{\mathbb C} $ is a specialization. Then $\theta(Res[A(x),B(x);x])=0$ iff one of the following two cases happens:
\begin{enumerate}
\item[(1)]  there exists an $x_0\in{\mathbb C}$ which is a root of both $\theta(A(x))$ and $\theta(B(x))$.
\item[(2)] the specialization reduces the degree of both $A(x)$ and $B(x).$\end{enumerate}
\end{prop}

See \cite{usingAG} for details on the above definition, and justifications of the above  facts about the resultant. 
We are primarily interested in the case  when $A(x)$ and $B(x)$ are both {\em monic general polynomials}, that is a polynomial with transcendental coefficients of the form $x^n+a_1x^{n-1}+ \cdots  + a_n$.  In this case a specialization does not reduce the degree.

Resultants may be  used  to eliminate variables in systems of polynomial equations. This leads to a method to calculate the root ratio polynomial. It is straightforward from the definitions to show:

\begin{prop} The root ratio polynomial, $R_p(r),$ of $p(x)= x^n + a_1 x^{n-1}+\cdots + a_n$ satisifies
  $$Res(p(rx),p(x);x) = a_n(r-1)^nR_p(r).$$\end{prop} 
  
  \begin{eg}
  Let $p(x) = x^2 + bx + c$, as in example \ref{eg_root_ratio}.  
  
  $$\begin{array}{rcl}Res(p(rx),p(x),x) & = & c^2 r^4-b^2 c r^3+(2b^2c - 2 c^2) r^2 - b^2 c r+c^2 \\ &= &c (r-1)^2(c r^2 + (2c-b^2)r + c) \\ & = & c (r-1)^2 R_p(r) \end{array}$$
  \end{eg}

\begin{dfn} Given two monic polynomials $p(x)$ and $q(x)$ the {\em common root ratio polynomial} is defined by \[C_{p,q} = Res(R_p(r),R_q(r);r).\]
\end{dfn}

The coefficients of $R_p(r)$ are polynomial in the coefficients of $p,$ therefore  $C_{p,q}$ is polynomial in the coefficients of $p$ and $q$.

\begin{eg}
The common root ratio polynomial for $p(x) = x^2 + ax + b$ and $q(x) = x^2 + cx + d$ is

\[C_{p,q}= (bc^2-d a^2)^2 .\]
The common root ratio polynomial for $p(x)=x^3+x^2+x+1$ and $q(x)=x^3+x^2+cx+d$ is
$$C_{p,q}=(c-d)^4(-c^2+2c^3+2d-4cd+d^2)^4.$$
\end{eg}

 The following is immediate:
\begin{prop}\label{cpq=0} Suppose $p(x),q(x)\in{\mathbb C}[x]$ are monic and both have degree at least $2$.  If $\alpha\in{\mathbb C}$ satisfies $R_p(\alpha)=R_q(\alpha)=0$ then $C_{p,q}=0.$
\end{prop}

\begin{prop}\label{prop:root_ratio}
Suppose $p(x),q(x)\in K[x]$ are monic polynomials and $C_{p,q}\in K$ is their common root ratio polynomial. Suppose $\theta:K\rightarrow{\mathbb C}$ is a specialization and $\theta(C_{p,q})=0.$ Then one of the following occurs  \begin{enumerate}
\item[(a)]  there is $r\in{\mathbb C}\setminus 0$ which is a root ratio of both $\theta(p(x))$ and $\theta(q(x)).$ 
\item[(b)]   $\theta(p(0))=0$ or $\theta(q(0))=0.$ 
\end{enumerate}
\end{prop}

\begin{proof} Applying \ref{resultantcases} to $\theta(C_{p,q})=\theta(Res[R_p(r),R_q(r);r])$ yields two possibilities. The first possibility, case (1), is that there is $r\in{\mathbb C}$ such that $\theta(R_p(r))=0=\theta(R_q(r)).$  Since $p$ is monic, the specialization does not reduce the degree of $p(x)$, and $\theta(R_p)=R_{\theta(p)}$.  Similarly for $q(x)$, therefore  $R_{\theta(p)}(r)=R_{\theta(q)}(r)=0.$  Case (b) results if either $\theta(p(x))$ or $\theta(q(x))$ has a root of $0$.

The remaining possibility is case (2) of \ref{resultantcases}; namely that the specialization reduces the $r$-degree of both $R_p(r)$ and $R_q(r).$ Since $p$ is monic, the coefficient of the highest power of $r$ in $R_p(r)$ is $a_n^{n-1}$ where $a_n=p(0)$ and $n=\deg_x(p(x)).$ Therefore specialization reduces the $r$-degree of $R_p(r)$ iff  $\theta(p(0))=0$ which implies conclusion (b).
\end{proof}

\begin{dfn}\label{dfn:eigen_ratio}
Given a matrix $M \in GL(n,{\mathbb C})$ let $\lambda_1,\cdots,\lambda_n$ be the eigenvalues of $M$ counted with multiplicity. For $1\le i\ne j\le n$ the numbers $\lambda_i/\lambda_j$ are called the  {\em eigenvalue ratios} of $M.$ 
\end{dfn}

Note that for any non-zero complex number $k$, the matrices $M$ and $k  M$ have the same eigenvalue ratios, hence eigenvalue ratio is well-defined for elements of $ \pglnc$.

\begin{lemma}\label{evrr} Let $P$ be the characteristic polynomial of a matrix $M\in GL(n,{\mathbb C})$ with $n\ge2.$ Then 
the set of eigenvalue ratios of $M$ equals the set of  zeroes of the root-ratio polynomial $R_P.$\end{lemma} 
\begin{proof} The characteristic polynomial of $P$ is monic and zero is not a root. The result now follows from \ref{rootratios}.\end{proof}

\begin{prop}\label{commonevratio}
For $n\ge 2$ let $A \subset \{(M,N) \,\, | \,\, M,N \in \slnpm\}$ be the set of pairs of matrices such that $M$ and $N$ have a common eigenvalue ratio.  Then $A$ is a real algebraic proper subvariety.
\end{prop}

\begin{proof} Since $\slnpm$ is defined by the polynomial equation $(\det M)^2=1$ it follows that $\slnpm\times\slnpm$ is a real algebraic variety. We will show $A$ is a subvariety by exhibiting it as the zero locus of one additional polynomial. It is clear that $A\ne\slnpm\times\slnpm$ thus the subvariety is proper. 

 Let $P(x)$  and $Q(x)$ be the characteristic polynomials of $M$ and $N$.  The common root ratio polynomial, $C_{P,Q},$ is a polynomial function of the entries of $M$ and $N$.  Indeed,  since $P$ and $Q$ are monic,  the common root-ratio polynomial  $C_{P,Q}$ is a polynomial in the coefficients of $P$ and $Q.$ These,  in turn, are polynomial functions of the entries of $M$ and $N$.  
 
 We claim that $A$ is the subvariety $V\subset\slnpm\times\slnpm$ defined by $C_{P,Q}=0.$ First assume that $(M,N)\in V.$  We apply proposition \ref{prop:root_ratio} to $C_{P,Q}.$   Since $M$ and $N$ are in $\slnpm$  it follows that $P$ and $Q$ have constant terms $\pm\det(M)$ and $\pm\det(N)$ that are non-zero, therefore case (b) does not occur. It now follows from case (a) that $P(x)$ and $Q(x)$ have a common root-ratio, which in turn implies the matrices $M$ and $N$ have a common eigenvalue ratio,  thus $V\subset A.$

Conversely, suppose $(M,N)\in A$, then $M$ and $N$ have a common eigenvalue ratio $\alpha.$ By \ref{evrr} it follows that $\alpha$ is a zero of the root-ratio polynomials $R_P$ and $R_Q.$  Proposition  \ref{cpq=0} implies that $C_{P,Q}$ is zero at $(M,N).$ Thus $(M,N)\in V$ and $A\subset V.$\end{proof}

\begin{dfn}\label{proximal_pair}
For $n\ge 2$ let $Q \subset \slnpm \times\slnpm$  be the union of $(I,I)$ with the set of pairs $(M,N)$ of proximal matrices  that  have the same Hilbert translation length.  The {\em proximal pair} variety is the Zariski closure of $Q$.
\end{dfn}

  \begin{cor}\label{samelength}  The proximal pair variety is a proper subvariety of $\slnpm \times \slnpm$ that does not contain $SO(n,1) \times SO(n,1)$.
  \end{cor}
\begin{proof} The Hilbert length of a proximal matrix is determined by the ratio of the largest and smallest eigenvalue by \ref{lemma:Hilbert_length}.  Thus the variety $A$ from proposition \ref{commonevratio} contains $Q$.   There are pairs of elements in $SO(n,1)$ with no common  eigenvalue ratios, hence $A$ does not contain $SO(n,1) \times SO(n,1)$.
\end{proof}

\section{Dual and Self-Dual Projective Structures}\label{section:proof_of_converse}
In this section we define the notion of dual projective structure and show that a structure is self-dual iff it is a hyperbolic structure.

An {\em ellipsoid} in $\rpn$ is the image under a projective map of the open unit ball, $B,$ in an affine patch. The subgroup of $\pglnone$ stabilizing $B$ is $Aut(B)\cong \ponone.$ The Hilbert metric on $B$ is isometric to hyperbolic space ${\mathbb H}^n$ and if $\Gamma$ is a discrete subgroup of $Aut(B)$ then $B/\Gamma$ is a (real) hyperbolic orbifold.

\begin{theorem}[Benoist,Thm 3.6 \cite{automorphs_benoist}]\label{theorem:benoist}
Let $\Gamma$ be a discrete subgroup of $\slnone$ that divides a properly convex open set $\Omega$ in $\rpn$. The Zariski closure $\overline{\Gamma}$ is $\sonone$ iff  $\Omega$ is an ellipsoid  and otherwise  $\overline{\Gamma}=SL(n+1,\R)$.
\end{theorem}

To allow for non-orientable orbifolds, we need the following corollary to Benoist's Theorem.

\begin{cor}\label{benoist_cor} Let $\Gamma$ be a discrete subgroup of $\slnonepm$ that divides a strictly convex open set $\Omega$ in $\rpn$. The Zariski closure $\overline{\Gamma}$ is either  $O(n,1)$ or $SO(n,1)$ iff  $\Omega$ is an ellipsoid.  Otherwise  $\overline{\Gamma}$ is either $\slnonepm$ or $\slnone$. \end{cor}

\begin{proof} Let $\Gamma_+ = \Gamma \cap SL(n+1,\R)$.   Then $\Gamma_+$ is at most an index two subgroup of $\Gamma$, and therefore divides $\Omega$.  By \ref{theorem:benoist}, the Zariski closure of $\Gamma_+$ is  $SO(n,1)$ when $\Omega$ is an ellipsoid and $SL(n+1,\R)$ otherwise.  When $\Omega$ is an ellipsoid the full stabilizer of $\Omega$ in $\slnonepm$  is $O(n,1)$.   Note that $SO(n,1)$ and $SL(n+1,\R)$ are respectively index two subgroups of $O(n,1)$ and $\slnonepm$.  Since the Zariski closure of a subgroup of an algebraic group is an algebraic subgroup (pg. 99 of \cite{onischik_vinberg}), the result follows.
\end{proof}

Choose an inner product $\left<,\right>$ on $\rnone.$ Let $\Omega\subset \rpn$ be a strictly convex open set. Let $\Omega_+\subset \sn$ be one of the pre-images  of $\Omega.$ Then
$$\{\ v\in\rnone\ :\ \forall x\in\Omega_+\ \left<v,x\right>\ > 0 \ \}$$ is a convex set and its image in $\rpn$ is called the {\em dual} of $\Omega$ and is denoted by $\Omega^*.$ A different choice of inner product results in a projectively equivalent domain. For convenience we will use the standard inner product.

It is immediate that if $A\in\pglnone$ preserves $\Omega$ then the transpose $A^t$ preserves $\Omega^*.$ The map
$$d:\pglnone\rightarrow\pglnone$$  given by $d(A)=(A^t)^{-1}$ is called the {\em global Cartan involution} or  the {\em dual map} and $d(A)$ is the {\em dual matrix.} It is an isomorphism and a different choice of inner product changes $d$ by composition with an inner automorphism.

It follows that the restriction of duality gives an isomorphism 
$$d_{\Omega}:Aut(\Omega)\rightarrow Aut(\Omega^*).$$
If $\Gamma < Aut(\Omega)$ is a discrete group then $Q=\Omega/\Gamma$ is a projective orbifold. The {\em dual orbifold} is $Q^*=\Omega^*/\Gamma^*$ where $\Gamma^*=d(\Gamma)$ is called the {\em dual group} to $\Gamma.$ 

\begin{prop} There is an orbifold isomorphism $f:Q\rightarrow Q^*$ such that $$d\circ hol_{Q}= hol_{Q^*}\circ f_*.$$\end{prop}
\begin{proof} The duality map provides an isomorphism $\pi_1^{orb}Q\rightarrow\pi_1^{orb}Q^*,$ however we must show they are orbifold isomorphic (if they are manifolds this means diffeomorphic). There is a natural diffeomorphism $f:\Omega\rightarrow\Omega^*$  defined using {\em affine spheres,} see for example \cite{loftin}. Naturality ensures that it is equivariant with respect  to the group actions.
\end{proof}

It follows that duality induces an involution on $C(Q)$ the space of projective structures on an orbifold $Q.$ A projective structure is called {\em self-dual} if it is projectively equivalent to is dual. Thus the equivalence classes of the self-dual structures are just the fixed points of this involution.

\begin{theorem}\label{theorem:converse}
Let $Q$ be a closed,  $n$-orbifold, where $n \geq 2$.  Then the only self-dual projective structures on $Q$ are hyperbolic.
\end{theorem}
\begin{proof} It is easy to see that hyperbolic orbifolds are self-dual. For the converse,   let $Q=\Omega/\Gamma$ be a self-dual  projective orbifold with dual $Q^*=\Omega^*/\Gamma^*.$  We may assume $\Gamma,\Gamma^*\subset\slnonepm.$ Self-duality implies there  is an element in $P\in SL(n+1,\R)$ such that for all $A\in\Gamma$ 
$$d(A)=(A^t)^{-1}=PAP^{-1}.$$   We claim that $\Gamma_+=\Gamma\cap \slnone$ is not Zariski dense in $\slnone.$ Indeed, this equation implies $f(A)=trace(A)-trace(A^{-1})$ is zero on $\Gamma_+$ and this is a non-trivial polynomial condition on $\slnone$  provided $n\ge 2.$ This proves the claim.

We now apply    theorem \ref{theorem:benoist}   to deduce that $\Omega$ is an ellipsoid. An ellipsoid can be mapped to the unit ball by a projective map, thus $\Gamma$ is conjugate by this map to a group which stabilizes the unit ball. The stabilizer of the unit ball is $O(n,1)$ and therefore $Q$ is equivalent to a hyperbolic orbifold.\end{proof}

Projective duality is due to the presence of the duality map on $\slnone.$  Since we have been unable to find a reference for the following theorem, we provide a proof.

\begin{theorem}[Automorphisms of $SL(n,{\mathbb C})$]\label{SLautomorphism}
For all $n\ge 2$ the group of holomorphic outer automorphisms of $SL(n,{\mathbb C})$ has order $2$ and is generated by the duality map, $d,$ that sends $A\mapsto (A^t)^{-1}.$ If $\sigma$ is an automorphism of $SL(n,{\mathbb C})$ then either $\sigma = \phi_g$ where $\phi_g$ is some inner automorphism, or  $\sigma = \phi_g \circ d$.  \end{theorem}

\begin{proof}
Let $\phi$ be a holomorphic automorphism of $\slnc.$  Differentiating  $\phi$  induces an automorphism on the tangent space at the identity, the Lie algebra $\mathfrak{sl}(n,{\mathbb C})$.  Furthermore, since $\slnc$ is connected, the differential map $D: Aut(\slnc) \ra Aut(\mathfrak{sl}(n,{\mathbb C}))$ is injective (\S 2 Theorem 4 \cite{onischik_vinberg}).  The inner automorphisms of $\slnc$ map onto the inner automorphisms, called the inner derivations, of $\mathfrak{sl}(n,\C)$.  

For a complex Lie algebra $G$, the group $Inn(G)$ forms a normal subgroup of the group of automorphisms, and $Aut(G)/Inn(G)$ is isomorphic to the automorphisms group of the Dynkin diagram for $G$ (pg. 202  \cite{onischik_vinberg}).   In the case of $\mathfrak{sl}(n,{\mathbb C})$, the Dynkin diagram is $A_n$, which is a linear graph with $n$ vertices and therefore has one nontrivial automorphism.  Since the map $Aut(\slnc) \ra Aut(\mathfrak{sl}(n,{\mathbb C}))$ is injective, $Aut(\slnc)/Inn(\slnc)$ has at most two elements.  The map $A \mapsto (A^T)^{-1}$, after differentiating, induces the non-trivial automorphism of the dynkin diagram.
\end{proof}  

Recall that $\glnc$ is a connected Lie group and therefore $\pglnc$ is also connected.  The natural map $\slnc\rightarrow\pglnc$ is an $n$-fold cyclic covering: the kernel is the center of $\slnc,$ which is isomorphic to the cyclic subgroup of order $n$ in $\C$ generated by a primitive $n$'th root of unity.

\begin{cor}[Automorphisms of $\pglnc$]\label{pglautomorphism}
For all $n\ge2$ the group of holomorphic outer automorphisms $Out(\pglnc)\cong{\mathbb Z}_2$ is generated by the duality map. 
\end{cor}
\begin{proof} 
The group $\slnc$ is  a finite cover of $\pglnc$ and therefore the Lie algebras of $\pglnc$ and $\slnc$ are isomorphic. Arguing as in \ref{SLautomorphism}  shows that $Out(\pglnc)$ is cyclic of order two, generated by the duality map. 
\end{proof}

We record some facts about normal subgroups in Lie groups  in the following lemma.

\begin{lemma}\label{normal_subgroup}
Every nontrivial normal subgroup of  $\pglnc$, $\ponc$, or\\ $\psonc$ contains $\psonc$.
\end{lemma}

\begin{proof}
The group $\pglnc\cong\pslnc$  has no non-trivial proper normal subgroups (page 227 \cite{rotman}).    

The group $\onc$ has two connected components, the one containing the identity is $\sonc$ (pg. 43 \cite{onischik_vinberg}).  Since $\onc$ is a simple Lie group with finite center, $\ponc$ is a centerless simple Lie group with identity component equal to $\psonc$.  The only proper normal subgroups in simple Lie groups are discrete.  By \cite{onischik_vinberg} pg. 45, any discrete normal subgroup of a connected Lie group is contained in the center.  Therefore, in all three cases, any nontrivial normal subgroup contains $\psonc$.
\end{proof}

\section{Proof of Main Theorem}\label{section:main}

The main theorem of this paper is that the Hilbert translation length function determines a strictly convex real projective structure up to projective duality.  We now state the main theorem more  precisely.

Suppose $Q$ is a closed strictly convex projective orbifold. The  translation length function defined in section \ref{section:background} is a map $\ell:\pi_1^{orb}Q\rightarrow{\mathbb R}.$
An element  $\alpha\in\pi_1^{orb}Q$ has finite order iff it is elliptic iff it has translation length zero. Otherwise the holonomy of $\alpha$ is  a proximal element $[A]\in\pglnone$ and 
$\ell(\alpha)=\log\left(\lambda_+/\lambda_-\right)$ where $\lambda_{\pm}$ are the eigenvalues of largest and smallest absolute value for any representative matrix  $A.$

\begin{theorem}\label{theorem:main}
Suppose that $Q$ is a closed  orbifold of dimension $n\ge 2.$
Let $$L: C(Q) \ra \R^{\pi_1^{orb}(Q)}$$ be the map which sends a strictly convex projective structure to its Hilbert translation length function.   For any $x \in C(Q)$,  we have $L^{-1}(L(x))=\{\ x,x^*\ \},$ where $x^*$ denotes the dual of $x$. Moreover $x=x^*$ iff $x$ is projectively self dual iff $x$ is a hyperbolic structure.
\end{theorem}

\begin{proof} For $i=1,2$ suppose that $x_i\in C(Q)$ are two strictly convex projective structures on $Q.$ These determine two holonomy representations $r_i:\pi_1^{orb}Q\rightarrow\Gamma_i\subset\pglnone$ up to conjugacy in $\pglnone.$ 

We will show that there is an automorphism $\phi$ of $\pglnone$ such that $r_2=\phi\circ r_1.$ By \ref{SLautomorphism} it follows that $r_2$ is conjugate to either $r_1$ or to the dual $d\circ r_1.$ By \ref{holgivesstructure} the conjugacy class of the holonomy determines the projective structure and  the result then follows.

By \ref{lemma:lift} we can always lift the holonomy to $\slnonepm.$  To prevent proliferation of notation, for the next few paragraphs we shall use $r_i$ for the lifted holonomy and $\Gamma_i$ for its image. 

Since $Q$ is a compact orbifold, there are only finitely many different orders of torsion elements. We define $k$ to be  the least common multiple of these orders. It follows that for every $g\in\pi_1^{orb}Q$ that $\left(r_i(g)\right)^k$ is either $I$ or proximal.  

Let $P\subset \slnonepm \times \slnonepm$ be the proximal pair variety defined in \ref{proximal_pair}. Then $P$ contains all pairs of proximal matrices with the same translation length. It also contains $( I,I).$ The map 
$$\tau:\slnonepm \times \slnonepm\rightarrow \slnonepm \times \slnonepm$$
given by $\tau(M,N)=(M^k,N^k)$ is a polynomial map therefore $A=\tau^{-1}(P)$ is an affine algebraic variety. 

Given $g\in\pi_1^{orb}Q$ the elements $r_1(g)$ and $r_2(g)$ have the same translation length. If the translation length is zero then $g$ is torsion, so  $r_1(g)^k$ and $r_2(g)^k$ are both the identity.  Otherwise, lemma \ref{lemma:Hilbert_length} implies that $r_1(g)$ and $r_2(g)$ have a common eigenvalue ratio r and therefore  $r_1(g)^k$ and $r_2(g)^k$ have a common eigenvalue ratio of $r^k$.  In either case, $(r_1(g),r_2(g))\in A.$ Thus
\[Im(r_1\times r_2) = \{(r_1(g),r_2(g)) \, | \, g \in \pi_1^{orb}Q\} \subseteq A\subseteq  \slnonepm \times \slnonepm.\] 

Algebraic geometry is easier over ${\mathbb C}$ than over ${\mathbb R}.$ For this reason we will now complexify everything. The {\em complexification} of a real affine algebraic variety, $V\subset{\mathbb R}^n,$ is the complex affine variety $V_{\mathbb C}\subset{\mathbb C}^n$ obtained by taking the complex zeros of the real polynomials defining  $V.$ Observe that $V= {\mathbb R}^n\cap V_{\mathbb C}.$

Define $D\subset\slnonepmc\times\slnonepmc$ to be the Zariski closure over ${\mathbb C}$ of $Im(r_1\times r_2).$ Since $D\subset A_{\mathbb C}$ it is a proper subvariety of $\slnonepmc\times\slnonepmc$. Let $\overline{\Gamma}_i$ be the Zariski closure over ${\mathbb C}$ of $\Gamma_i$ in $\slnonepmc.$ By corollary \ref{benoist_cor}, $\overline{\Gamma}_i$ is either conjugate to $SO(n,1,{\mathbb C})$ or $O(n,1,{\mathbb C})$, or else equals $\slnonec$ or $\slnonepmc$.  In the former cases, these groups are isomorphic to $SO(n+1,{\mathbb{C}})$ and $O(n+1,{\mathbb C}),$   and we adjust by a conjugacy so that $\overline{\Gamma}_i= \sononec$ or $\ononec.$ Henceforth, in all cases we may assume that $\sononec\subseteq\overline{\Gamma}_i.$

Since $\tau(\sononec \times \sononec) = \sononec \times \sononec$, corollary \ref{samelength} implies  $A,$ and therefore $D,$ does not contain $\sononec\times \sononec.$  We record this fact for later reference.

\begin{remark}\label{observation}
$D$ does not contain $\sononec\times \sononec.$  
 \end{remark}
 
The Zariski closure of a subgroup of an algebraic group  is an algebraic subgroup (Lemma 2.2.4 of \cite{LAG}) therefore $D$ is a complex algebraic subgroup.  For $i=1,2$  let $p_i: D \ra \slnonepmc$ be the projection map onto the $i$'th coordinate. By (theorem 3, page 102 of \cite{onischik_vinberg}) the image under an algebraic homomorphism of a complex algebraic group is also a complex  algebraic group. It follows that $G_i=p_i(D)$ is a complex algebraic subgroup of $\slnonepmc.$ 

Now $G_i$ contains $\Gamma_i$ and $G_i$ is  algebraic so $\overline{\Gamma}_i\subseteq G_i.$  We claim that $$\overline{\Gamma}_i=G_i.$$ This is because $p_i^{-1}(\overline{\Gamma}_i)$ is a complex algebraic group which contains $Im(r_1\times r_2),$ thus $D\subset p_i^{-1}(\overline{\Gamma}_i)$ and $$G_i=p_i(D)\subseteq p_i(p_i^{-1}(\overline{\Gamma}_i))=\overline{\Gamma}_i.$$ 
Now we projectivize everything.
Let $P:\slnonepmc\rightarrow\pglnonec$ be the natural projection. Define $$PD=\(P\times P)(D)\subset\pglnonec\times\pglnonec.$$
 We have the two coordinate projections  $\pi_i:PD\rightarrow \pglnonec.$ Define $$PG_i=\pi_i(PD)=P(G_i)=P(\overline{\Gamma}_i)\subset\pglnonec.$$   We have shown that $PG_i$ is either $\pglnonec$,
  $\pononec$, or $\psononec$.

\medskip\noindent{\bf Claim.}  $\pi_1$ and $\pi_2$ are injective.

\medskip Suppose that $\pi_1$ is not injective. Then $\ker(\pi_1)$ is a non-trivial normal subgroup of $PD.$    Now $\ker(\pi_1)\subset 1\times \pglnonec$ and $\pi_2|1\times \pglnonec$ is an isomorphism.  Applying $\pi_2$ we see that  $\pi_2(\ker(\pi_1))$ (which is isomorphic to $\ker(\pi_1)$) is a non-trivial normal subgroup of  $\pi_2(PD)=PG_2.$  

Recall $PG_2$ is one of $\psononec$, $\pononec$ or $\pglnonec.$    By lemma \ref{normal_subgroup} any non-trival normal subgroup of these groups contains $\psononec$.  It follows that  $\pi_2(\ker(\pi_1))\supseteq \psononec$ thus $\ker (\pi_1)\supseteq 1\times \psononec.$ Since $PD$ contains this group it also contains $P\Gamma_1\times \psononec$ and therefore contains $P\overline{\Gamma}_1\times \psononec.$ Since $P\overline{\Gamma}_1=PG_1$ and $PG_1\supseteq \psononec$ it follows that $PD$ contains $\psononec \times \psononec.$ But this contradicts remark \ref{observation} which proves the claim.

\medskip

We now have an isomorphism $\phi=\pi_2\circ \pi_1^{-1}:PG_1\rightarrow PG_2$ which maps $\Gamma_1$ onto $\Gamma_2.$ Therefore $PG_1 = PG_2.$  If $PG_1=\psononec$ or $\pononec$ this implies both projective structures are hyperbolic (follows from  \ref{benoist_cor} and complexifying). Since $Q$ is closed, Mostow rigidity (Theorem $A'$ page 4 of \cite{mostow})  implies  the two structures are isometric, and thus projectively isomorphic by \ref{theorem:not_isometric}.

If $G_1=\pglnonec$  then $\phi$ is either an inner automorphism, or the composition of the duality map and an inner automorphism by \ref{pglautomorphism}.  The assertion of the theorem involves $\pglnone$ rather than $\pglnonec.$ We need to prove the inner automorphism of $\pglnonec$  actually extends an inner automorphism of  $\pglnone.$

Suppose  an element $A\in\glnonec$ conjugates $\Gamma_1$ to $\Gamma_2.$ Since the latter subgroups are real  $A$ may be chosen to have real entries.  Then we obtain an inner automorphism of $\pglnone.$ A similar remark applies in the case of the dual map composed with an inner automorphism. In either case,  since $\Gamma_i\subset G_i$, restriction gives $r_2=\phi\circ r_1$.

\end{proof}

\begin{cor}
There exist two  diffeomorphic strictly convex real projective surfaces $M$ and $N$ with equivalent marked length spectra that are not isometric.
\end{cor}

\begin{proof}
Goldman shows that the dimension of the deformation space of strictly convex projective structures on a surface $S$ is $-8 \chi(S)$.  This implies that there exists non-hyperbolic strictly convex projective structures.  Together theorems \ref{theorem:converse} and \ref{theorem:main} imply there are non-equivalent projective structures on $S$ with the same marked length spectra.  The corollary then follows from proposition \ref{theorem:not_isometric}.
\end{proof}

As an example consider the $2$-orbifold $Q=S^2(3,3,4).$ According to Goldman and Choi, $C(Q)\cong\R^2.$ There is a unique hyperbolic structure on this orbifold, which we may assume is $0\in\R^2$, and the duality map induces an involution on $C(Q)$ which is a rotation of order $2$ fixing $0.$ All the non-zero points correspond to projective structures that are not self dual.

A closed hyperbolic $n$-manifold which contains an embedded, totally geodesic codimension-$1$ submanifold admits a $1$ parameter family of strictly convex projective structures obtained by a projective version of the bending construction. Thus isospectral examples also exist in dimensions $3$ and $4.$

\end{document}